\documentclass[a4paper,12pt,reqno]{amsart}
\usepackage{a4wide}
\usepackage{amstext}
\usepackage{ucs}
\usepackage{amsmath}
\usepackage{amsfonts}
\usepackage{amssymb}
\usepackage{amsthm}
\usepackage{caption}
\usepackage{subfigure}
\usepackage{upgreek}
\usepackage{xcolor}
\usepackage{mathtools}
\usepackage[UKenglish,USenglish]{babel}
\usepackage[T1]{fontenc}
\usepackage{etoolbox}
\usepackage{esint,commath,ifthen}
\usepackage{graphics}
\usepackage{enumitem}
\usepackage{faktor}
\usepackage[normalem]{ulem}
\usepackage{hyperref}
\hypersetup{
    colorlinks=true,
    linkcolor=blue,
    citecolor=magenta
    }
\usepackage{mathrsfs}
\usepackage{tikz-cd}

\newtheorem{theorem}{Theorem}[section]
\newtheorem{lemma}[theorem]{Lemma}
\newtheorem{proposition}[theorem]{Proposition}
\newtheorem{corollary}[theorem]{Corollary}
\theoremstyle{definition}

\newtheorem{remark}[theorem]{Remark}

\allowdisplaybreaks[4]

\title[Bounded orbits of $\mathrm{SL}_3(\mathbb{C})/\mathrm{SL}_3(\mathcal{O}_{\mathbb{K}})$]{Hyperplane absolute winning property of bounded orbits under diagonalizable flows on $\mathrm{SL}_3(\mathbb{C})/\mathrm{SL}_3(\mathcal{O}_{\mathbb{K}})$}
\author{Gaurav Sawant}
\address{Department of Mathematics, The Institute of Science, 15, Madam Cama Road, Mumbai 400032, India}
\email{gauravsawant.math@gmail.com}
\date{}

\begin{document}

\begin{abstract}

We extend the work of An, Guan and Kleinbock on bounded orbits of diagonalizable flows on $\mathrm{SL}_3(\mathbb{R})/\mathrm{SL}_3(\mathbb{Z})$ to $\mathrm{SL}_3(\mathbb{C})/\mathrm{SL}_3(\mathcal{O}_{\mathbb{K}})$, where $\mathbb{K}$ is an imaginary quadratic field.  To achieve this, we first prove a complex analogue of Minkowski's Linear Forms Theorem.  We then set up an appropriate Schmidt game in $\mathbb{C}^3$ such that bounded orbits correspond to a hyperplane-absolute-winning set consisting of certain vectors in $\mathbb{C}^3$ relative to an approximation by imaginary quadratic rationals in $\mathbb{K}$.
\end{abstract}
\maketitle

\section{Introduction}

In the seminal paper \cite{Sch22}, W. Schmidt devised a game for which badly approximable numbers constitute a winning set.  Schmidt's game has been a powerful tool for proving that certain sets of badly approximable numbers have full Hausdorff dimension.  Following the work of McMullen \cite{McM} on absolute and potential games, the more refined notion of hyperplane-absolute game on $\mathbb{R}^n$ was introduced in \cite{BFKRW}.  This notion has been successfully used to study the relation between a certain exceptional set of vectors in $\mathbb{R}^3$ and bounded orbits of diagonalizable flows on $\mathrm{SL}_3(\mathbb{R})/\mathrm{SL}_3(\mathbb{Z})$ in \cite{AGK}.

Geometry of numbers, originating from the work of Minkowski \cite{Minkowski1896}, has profoundly influenced the field of Diophantine approximation (See also \cite{Cassels}).  Building over this work, Siegel \cite{Siegel} proved a powerful existence theorem for systems of linear equations with integer coefficients, now known as Siegel's Lemma.  Bombieri and Vaaler \cite{BV83} extended this result to ad\`{e}les over an algebraic number field $k$.  In another direction, Minkowski's geometry of numbers was developed for algebraic number fields by Rogers and Swinnerton-Dyer \cite{RSD}, and to $S$-adic global fields by Kleinbock, Shi and Tomanov \cite{KST}. 

Rational approximations to complex numbers have been studied for well over a century.  A. Hurwitz \cite{Hur1} introduced a generalization of continued fractions to complex numbers over the discrete subrings $\mathbb{Z}[i]$ and $\mathbb{Z}[\sqrt{-3}]$.  Approximations of complex numbers were further explored in detail by Lakein \cite{Lakein} and A. Schmidt \cite{ASchmidt} (see also \cite{Dani1,Dani2}).  Building over these works, Hines \cite{Hines} has explicitly determined the bounds on coefficients for continued fraction approximations relative to Euclidean imaginary quadratic fields.  In another direction, following Baker \cite{Baker} and Stark \cite{Stark}, Esdahl-Schou and Kristensen \cite{ESK} showed that the badly approximable complex numbers relative to the class-number one number fields form a set of full Hausdorff dimension.

Recently, modern methods such as the Dani correspondence \cite{Dani86} have been used to study badly approximable complex numbers.  Badly approximable complex vectors relative to a number field were first explored in \cite{EGL}, where it was shown that on $C^1$ curves, they form a winning set, thus having full Hausdorff dimension.  In the same paper, Schmidt's conjecture on intersection of sets of badly approximable numbers (\cite{Sch82}; see also \cite{BPV}) was verified in the case of real quadratic number fields.  In \cite{S}, it was shown that the set of badly approximable complex numbers by ratios of elements of the ring of integers of totally imaginary number fields has full Hausdorff dimension (in fact, it is hyperplane-absolute-winning with respect to a suitable Schmidt game), building on the techniques developed in \cite{AGGL}.

In this paper, we set up a hyperplane-absolute game on $\mathbb{C}^3$ to study a weighted approximation of vectors in $\mathbb{C}^3$ relative to an imaginary quadratic number field $\mathbb{K}$.  Extending the methods of An, Guan and Kleinbock \cite{AGK} to the complex setting, we obtain the following analogue of \cite[Theorem 1.2, Theorem 1.3]{AGK}:

\begin{theorem}\label{EF}
 Let $G=\mathrm{SL}_3(\mathbb{C})$, and $\Gamma=\mathrm{SL}_3(\mathcal{O}_{\mathbb{K}})$.  Let $F^+ = \{g(t):t\geq 0 \}$ denote a one-parameter $\mathbb{R}$-diagonalizable subsemigroup of $G$.  Then the set
\[ E(F^+)=\{\Lambda\in G/\Gamma : F^+\Lambda~\text{is~bounded}\}\]
is hyperplane-absolute-winning (HAW) on $G/\Gamma$.  Moreover, for \[H=\{h\in G:\lim_{t\to\infty} (g(t))^{-1}hg(t)=e \},\] the set $\{h\in H : h\Lambda\in E(F^+)\}$ is HAW on $H$ for any $\Lambda\in G/\Gamma$.
\end{theorem}
In \S2, we prove a complex analogue of Minkowski's Linear Forms Theorem corresponding to the imaginary quadratic number field $\mathbb{K}$.  We begin \S3 by a correlation of the set of bounded trajectories of a particular dynamical action on $\mathrm{SL}_3(\mathbb{C})/\mathrm{SL}_3(\mathcal{O}_{\mathbb{K}})$ with a particular set of vectors in $\mathbb{C}^3$.  To achieve this, we consider the sets of vectors in $\mathbb{C}^3$ sufficiently away from the rational points relative to $\mathbb{K}$.  We then classify the closed-ball subsets of a given closed ball in $\mathbb{C}^3$ as well as the said rational points using a suitable height function.  We also show that the vectors in $\mathbb{C}^3$ in a closed-ball subset that are very close to the rational points in fact lie in an appropriate neighborhood of a complex hyperplane.  Using this observation, we set up in \S4 a hyperplane-potential game on $\mathbb{C}^3$ to show that this particular set of complex vectors is HAW.  Using this, we arrive at a proof of Theorem \ref{EF}.

\section{Preliminaries}
Consider an imaginary quadratic number field $\mathbb{K}=\mathbb{Q}(\sqrt{-d})$, where $d$ is a squarefree positive integer, and let $\mathcal{O}_{\mathbb{K}}$ denote its ring of integers.  Observe that there is only one archimedean place for $\mathbb{K}$, which is the field of complex numbers $\mathbb{C}$, so that $\mathbb{C}$ is an $S$-arithmetic space over $\mathbb{K}$.  In particular, the (only) two field embeddings $\mathbb{K}\hookrightarrow\mathbb{C}$ are given by $p\mapsto p$ (the identity mapping) and $p\mapsto\overline{p}$ (the conjugate mapping), so that the $\mathbb{K}$-norm of $p\in\mathbb{K}$ is $N_{\mathbb{K}}(p)=p\overline{p}=|p|^2$.
 \subsection{A Complex Version of Minkowski's Linear Forms Theorem}
Analogous to \cite[Theorem 5.3]{Neukirch} and \cite[Theorem 95]{Hecke}, we utilize the real version of Minkowski's Linear Forms theorem to state:
\begin{theorem}[Minkowski's Linear Forms Theorem for general lattices in $\mathbb{R}^n$]\label{MinkRLF}
 Let $\Lambda$ be a lattice in $\mathbb{R}^n$.  Consider real linear forms
 \[ L_j(x_1,\dots,x_n) = \sum_{k=1}^n a_{jk}x_k, \quad j=1,\dots,n \]
 such that $\det(a_{jk})\neq 0$, and let $\kappa_1,\dots,\kappa_n>0$ be such that $\displaystyle \prod_{j=1}^n\kappa_j\geq|\det(a_{jk})|\mathrm{covol}(\Lambda)$.  Then there exists a nonzero element $\mathbf{v}\in\Lambda$ for which
 \[ |L_j(\mathbf{v})|<\kappa_j,\quad j=1,\dots,n. \]
\end{theorem}
\begin{proof}
 Fix $g\in\mathrm{GL}_n(\mathbb{R})$ with $\Lambda=g\mathbb{Z}^n$ so that $\mathrm{covol}(\Lambda)=|\det(g)|$.  Then the theorem follows from Minkowski's Linear Forms Theorem for $\mathbb{Z}^n$ \cite[Theorem 94]{Hecke} by taking $L_j'=L_j\circ g$.
\end{proof}
Now we will regard $\mathcal{O}_{\mathbb{K}}^n$ as a lattice in $\mathbb{R}^{2n}$ through the identification $\mathbb{C}^n \simeq \mathbb{R}^{2n}$.  For this, we fix the integral basis $\{1,\omega\}$ of $\mathcal{O}_{\mathbb{K}}$ for $\omega=\frac{d_{\mathbb{K}}+\sqrt{-d_\mathbb{K}}}{2}$, where  $d_{\mathbb{K}}=\begin{cases}
 4d & \text{if}~d\equiv 1,2 \mod 4, \\
  d & \text{if}~d\equiv 3 \mod 4. 
  \end{cases}$ is the discriminant of $\mathbb{K}$.  In other words,
$\mathcal{O}_{\mathbb{K}}=                               \mathbb{Z}\oplus\mathbb{Z}\omega$.  Observe that every $z={x}+i{y}$ can be written as $z=\tilde{x}+\tilde{y}\omega$, where $ \begin{pmatrix} x \\ y \end{pmatrix} = \begin{pmatrix} 1 & \frac{{d_{\mathbb{K}}}}{2} \\ 0 & \frac{\sqrt{d_{\mathbb{K}}}}{2} \end{pmatrix}\begin{pmatrix} \tilde{x} \\ \tilde{y} \end{pmatrix}$.
Accordingly, we consider the lattice
\[ \Lambda_{\mathbb{K}}=                            \mathrm{diag}\left(\underbrace{\begin{pmatrix} 1 & \frac{{d_{\mathbb{K}}}}{2} \\ 0 & \frac{\sqrt{d_{\mathbb{K}}}}{2} \end{pmatrix},\dots,\begin{pmatrix} 1 & \frac{{d_{\mathbb{K}}}}{2} \\ 0 & \frac{\sqrt{d_{\mathbb{K}}}}{2} \end{pmatrix}}_{n~\text{times}}\right)\mathbb{Z}^{2n} \]
to propose the following complex version of Minkowski's Linear Forms Theorem:
\begin{proposition}[A complex version of Minkowski's Linear Forms Theorem]\label{MinkCLF}
 Let $\mathbb{K}=\mathbb{Q}[\sqrt{-d}]$, where $d$ is a squarefree positive integer.  Consider a system of $n$ complex linear forms
 \[ L_j(z_1,\dots,z_n)=\sum_{k=1}^n a_{jk}z_k,\quad (1\leq j\leq n) \]
 such that the associated $2n\times 2n$ matrix
 \[ L =  \left(\begin{pmatrix}
             \mathrm{Re}(a_{jk}) & -\mathrm{Im}(a_{jk}) \\ 
             \mathrm{Im}(a_{jk}) & \mathrm{Re}(a_{jk}) 
            \end{pmatrix}\right)_{1\leq j,k\leq n} \]
 satisfies $\det L \neq 0$.  Let $c_1,\dots,c_n>0$ be such that
 \[ \prod_{k=1}^n c_k>d_{\mathbb{K}}^{\frac{n}{4}}|\det L|. \]
 Then there is a nonzero element $\mathbf{v}\in\mathcal{O}_{\mathbb{K}}^n$ for which
 \[ |L_j(\mathbf{v})|<c_j.\quad (1\leq j\leq n). \]
\end{proposition}
\begin{proof}
 Each complex linear form $L_j$ gives rise to two real linear forms
 \[ \sum_{k=1}^n(\mathrm{Re}(a_{jk})x_k-\mathrm{Im}(a_{jk})y_k)\] and \[
             \sum_{k=1}^n (\mathrm{Im}(a_{jk})x_k + \mathrm{Re}(a_{jk})y_k), \]
corresponding to $z_k=x_k+iy_k$, $1\leq k\leq n$.  The statement of the proposition follows from Theorem \ref{MinkRLF} using the $2n$ real linear forms as above, with $\Lambda=\Lambda_{\mathbb{K}}$, and positive real numbers $\kappa_1,\dots,\kappa_{2n}$ given by
\begin{equation}\label{kap} \kappa_{2k-1}=\kappa_{2k}=\frac{c_k}{\sqrt{2}} \quad (1\leq k\leq n). \end{equation} 
\end{proof}

Letting $n=3$, we obtain the following corollary as an analogue of \cite[Lemma 3.3]{AGK}:

\begin{corollary}\label{AGKl3.3}
 Assume that $\mathbb{K}$ is as in Proposition \ref{MinkCLF}.  Let $z\in\mathbb{C}$.  For any $(p,r,q)\in \tilde{\mathcal{O}}^3_{\mathbb{K}}$, there exists $(a,b,c)\in\mathcal{O}_{\mathbb{K}}^3$ with $(a,b)\neq(0,0)$ such that
  \[ \overline{a}p+\overline{b}r+\overline{c}q=0,\quad |a|\leq{\sqrt[4]{4d}}|q|^\lambda, \quad |b+za|\leq{\sqrt[4]{4d}}|q|^{1-\lambda}. \]
\end{corollary}
\begin{proof}
Denote $p=p_1+ip_2$, where $p_1\in\mathbb{Z}+\mathbb{Z}\dfrac{{d_{\mathbb{K}}}}{2},p_2\in\mathbb{Z}\dfrac{\sqrt{d_{\mathbb{K}}}}{2}$, etc.  To apply Theorem \ref{MinkRLF}, we consider the $6\times 6$ matrix
\[ L=\begin{pmatrix}
      p_1&p_2&r_1&r_2&q_1&q_2\\
      -p_2&p_1&-r_2&r_1&-q_2&q_1\\
      1&&&&&\\
      &1&&&& \\
      z_1&-z_2&1&&&\\
      z_2&z_1&&1&&
     \end{pmatrix}
\]
whose rows correspond to six real linear forms, and whose determinant is $q_1^2+q_2^2=|q|^2\neq 0$.

As in Proposition \ref{MinkCLF}, we set $\Lambda=\Lambda_{\mathbb{K}}$, and set $n=3$.  Take positive numbers
\[ (\kappa_1,\kappa_2)=\begin{cases} (1,\sqrt{d}) & \text{if}~d\equiv 1,2\mod 4,\\ \left( \frac{1}{2},\frac{\sqrt{d}}{2}\right) & \text{if}~d\equiv 3\mod 4,\end{cases} \quad \kappa_3=\kappa_4=\sqrt[4]{4d}|q|^{\lambda},\quad \kappa_5=\kappa_6=\sqrt[4]{4d}|q|^{1-\lambda}. \]
In either case, we have $\prod_{j=1}^6\kappa_j > \frac{1}{8}{d_{\mathbb{K}}^{\frac{3}{2}}}|q|^2=\mathrm{covol}(\Lambda_{\mathbb{K}})\det L$.  Thus an argument similar to that in the proof of Proposition \ref{MinkCLF} applies to yield existence of a solution $(a,b,c)\in\mathcal{O}_{\mathbb{K}}^3$ such that
\begin{align*}
 |\mathrm{Re}(a\overline{p}+b\overline{r}+c\overline{q})|&<\begin{cases} 1 & \text{if}~d\equiv 1,2\mod 4,\\ \frac{1}{2} & \text{if}~d\equiv 3\mod 4,\end{cases} \\
 |\mathrm{Im}(a\overline{p}+b\overline{r}+c\overline{q})|&<\begin{cases} \sqrt{d} & \text{if}~d\equiv 1,2\mod 4,\\ \frac{\sqrt{d}}{2} & \text{if}~d\equiv 3\mod 4,\end{cases} \\
|a|\leq\sqrt[4]{4d}|q|^\lambda,\quad |b+za|&\leq\sqrt[4]{4d}|q|^{1-\lambda}.
\end{align*}
Identical to the argument in \cite[Lemma 3.3]{AGK}, we obtain $(a,b)\neq(0,0)$.
\end{proof}
\subsection{Vector Products in \texorpdfstring{$\mathbb{C}^3$}{C3}}

Recall that the standard Hermitian inner product on $\mathbb{C}^3$ is defined by
\begin{equation}\label{hermprod} \langle \mathbf{v},\mathbf{w} \rangle =\bar{v}_1w_1+\bar{v}_2w_2+\bar{v}_3w_3 \quad\text{for}~\mathbf{v}=(v_1,v_2,v_3),\mathbf{w}=(w_1,w_2,w_3)~\text{in}~\mathbb{C}^3. \end{equation}

Notationally analogous to the definitions on $\mathbb{R}^3$, we set the \emph{dot product} of $\mathbf{v},\mathbf{w}\in\mathbb{C}^3$ as
\[ \mathbf{v}\cdot\mathbf{w} = \langle \mathbf{v},\mathbf{w} \rangle. \]
We define the \emph{cross product} of $\mathbf{v},\mathbf{w}\in\mathbb{C}^3$ by
\begin{equation}\label{crossprod}
 \mathbf{v}\times\mathbf{w}=(\overline{v_2w_3}-\overline{v_3w_2}, \overline{v_3w_1}-\overline{v_1w_3}, \overline{v_1w_2}-\overline{v_2w_1}).
\end{equation}
The above definition satisfies conjugate-linearity in each of the components, antisymmetry, distributivity, self-annihilation, and the scalar triple product identity
\[ \mathbf{u}\cdot(\mathbf{v}\times\mathbf{w}) = \mathbf{v}\cdot(\mathbf{w}\times\mathbf{u}) = \mathbf{w}\cdot(\mathbf{u}\times\mathbf{v})~\text{for~all}~\mathbf{u},\mathbf{v},\mathbf{w}\in\mathbb{C}^3. \]
Moreover, it can be easily seen that the following vector triple product identity holds:
\begin{equation}\label{vec3prod}
 \mathbf{u}\times(\mathbf{v}\times\mathbf{w})=(\mathbf{u}\cdot\mathbf{w})\mathbf{v}-(\mathbf{u}\cdot\mathbf{v})\mathbf{w}~\text{for~all}~\mathbf{u},\mathbf{v},\mathbf{w}\in\mathbb{C}^3.
\end{equation}
We will use \eqref{hermprod}, \eqref{crossprod}, and \eqref{vec3prod} in the technical preparation in \S3.

\section{A Weighted Approximation in \texorpdfstring{$\mathbb{C}^3$}{C3}}

In order to approximate $\mathbf{z}=(\xi,\eta,\zeta)\in\mathbb{C}^3$ by triples of rational elements from $\mathbb{K}$, we employ the following terminology analogous to \cite{AGK}: 

Consider the subgroups
\begin{equation}\label{U}
 U=\left\{ u_{\mathbf{z}}~:~\mathbf{z}=(\xi,\eta,\zeta)\in\mathbb{C}^3 \right\} ,\quad\text{where}\quad u_{\mathbf{z}}:=
 \begin{pmatrix}
  1 & \zeta & \xi \\ 0& 1 & \eta \\ 0& 0& 1
 \end{pmatrix}
\end{equation}
and
\begin{equation}\label{U-tilde}
 \tilde{U}=\left\{ \tilde{u}_{\mathbf{z}}~:~\mathbf{z}=(\xi,\eta,\zeta)\in\mathbb{C}^3 \right\} ,\quad\text{where}\quad \tilde{u}_{\mathbf{z}}:=
 \begin{pmatrix}
  1 & 0 & \eta \\ \zeta & 1 & \xi \\ 0&0 & 1
 \end{pmatrix}
\end{equation}
of $G:=\mathrm{SL}_3(\mathbb{C})$, and let $\Gamma:=\mathrm{SL}_3(\mathcal{O}_{\mathbb{K}})$.  Note that the subgroups $U$ and $\tilde{U}$ are conjugate to each other through the permutation matrix
\begin{equation}\label{perm}
 P=\begin{pmatrix}
    0&1&0\\1&0&0\\0&0&1
   \end{pmatrix}.
\end{equation}
The mappings defined by
\begin{equation}\label{u_z-inv}
(\xi,\eta,\zeta)=\mathbf{z}\mapsto u_{\mathbf{z}}^{-1}=u_{(0,0,-\zeta)}u_{(-\xi,-\eta,0)} \end{equation}
and \begin{equation}\label{u_z-inv-tilde}
(\xi,\eta,\zeta)=\mathbf{z}\mapsto \tilde{u}_{\mathbf{z}}^{-1}=\tilde{u}_{(0,0,-\zeta)}\tilde{u}_{(-\xi,-\eta,0)} \end{equation}
are biholomorphic embeddings of $\mathbb{C}^3$ into $U$ and $\tilde{U}$, respectively (cf. \cite{Goldman,Wells}).  For a weight $0\leq\lambda\leq 1$, define the (normalized) diagonal flow  
\begin{equation}\label{F-lambdaplus}
F_\lambda^+=\{g_{\lambda}(t)~:~t\geq 0\}, 
\end{equation}
 where \begin{equation}\label{gt}
 g_{\lambda}(t)=\mathrm{diag}(e^{\lambda t},e^{(1-\lambda)t},e^{-t}).
 \end{equation}
Then, for $\mathbf{v}=(p,r,q),\tilde{\mathbf{v}}=(r,p,q)\in\mathcal{O}_{\mathbb{K}}^3$, the entries of the column vectors in $F_\lambda^+u_{\mathbf{z}}^{-1}\mathbf{v}^T$ (resp. $F_\lambda^+\tilde{u}_{\mathbf{z}}^{-1}\tilde{\mathbf{v}}^T$) correspond to a $\lambda$-weighted approximation of $\mathbf{z}$ relative to $\mathbf{v}$ for $\lambda>\frac{1}{2}$ (resp. relative to $\tilde{\mathbf{v}}$ for $\lambda<\frac{1}{2}$) as described below.

For $\varepsilon>0$ and $\mathbf{v}=(p,r,q)\in\tilde{\mathcal{O}}_{\mathbb{K}}^3:=\mathcal{O}_{\mathbb{K}}\times \mathcal{O}_{\mathbb{K}}\times ( \mathcal{O}_{\mathbb{K}}\setminus\{0\} )$, denote
\[ 
\Delta_{\varepsilon,\mathbb{K},\lambda}(\mathbf{v}) := \left\{ (\xi,\eta,\zeta)\in\mathbb{C}^3:\left| \xi-\dfrac{p}{q}-\zeta\left( \eta-\dfrac{r}{q} \right) \right|<\dfrac{\varepsilon\sqrt[4]{d}}{|q|^{1+\lambda}}, \left| \eta-\dfrac{r}{q} \right|<\dfrac{\varepsilon\sqrt[4]{d}}{|q|^{2-\lambda}} \right\}.
\]
Define
\begin{equation}\label{bad-epsilon}
 \mathrm{B}_{\varepsilon,\mathbb{K}}(\lambda) := \mathbb{C}^3\setminus  \bigcup_{\mathbf{v}\in \tilde{\mathcal{O}}^3_{\mathbb{K}}} \Delta_{\varepsilon,\mathbb{K},\lambda}(\mathbf{v})
\end{equation}
and
\begin{equation}\label{bad-lambda}
 \mathrm{B}_{\mathbb{K}}(\lambda) := \bigcup_{\varepsilon>0} \mathrm{B}_{\varepsilon,\mathbb{K}}(\lambda).
\end{equation}
The boundedness of the trajectories $F_\lambda^+u_{\mathbf{z}}^{-1}\Gamma$ in $G/\Gamma$ can be characterized analogous to \cite[Lemma 3.2]{AGK} as follows:
 \begin{proposition}\label{AGKl3.2}
 Let $\lambda\geq\frac{1}{2}$.  A vector $\mathbf{z}=(\xi,\eta,\zeta)\in\mathbb{C}^3$ is in $\mathrm{B}_{\mathbb{K}}(\lambda)$ if and only if the trajectory $F_{{\lambda}}^+u_{\mathbf{z}}^{-1}\Gamma$ is bounded in $G/\Gamma$; that is, there is $\varepsilon=\varepsilon(\mathbf{z})>0$ such that
 \begin{equation}\label{AGKe3.3}
\max \left\{ |q|^{\lambda}|q \xi-p-\zeta(q \eta-r)|, |q|^{1-\lambda}|q \eta-r| \right\} \geq\varepsilon
 \end{equation}
 for all $(p,r,q)\in \tilde{\mathcal{O}}^3_{\mathbb{K}}$.
\end{proposition}

\begin{proof}
Let $g_{\lambda}(t)$ be given by \eqref{gt}.  The number field analogue of Mahler's compactness criterion \cite[Theorem 1.1]{KST} implies that $F_{{\lambda}}^+u_{\mathbf{z}}^{-1}\Gamma$ is bounded if and only if there exists $0<\delta\leq 1$ such that, for any $t\geq 0$ and $(p,r,q)\in \tilde{\mathcal{O}}^3_{\mathbb{K}}$,
\[ \left\| g_{\lambda}(t)u_{\mathbf{z}}^{-1} \begin{pmatrix} p \\ r \\ q \end{pmatrix} \right\|_\infty\geq\delta; \]
that is, if and only if
\begin{equation}\label{AGKe3.4}
 \max \left\{ e^{\lambda t}|p-q\xi-\zeta(r-q\eta)|, e^{(1-\lambda)t}|r-q\eta|,e^{-t}|q| \right\} \geq\delta.
\end{equation}
Suppose first that \eqref{AGKe3.3} holds for some $\varepsilon>0$.  We will prove that \eqref{AGKe3.4} is true for \[\delta=\min\{ 1, \varepsilon^{\frac{1}{1+\lambda}},\varepsilon^{\frac{1}{2-\lambda}} \}.\]
Suppose on the contrary that there exist $(p,r,q)\in \tilde{\mathcal{O}}^3_{\mathbb{K}}$ and $t\geq 0$ for which
\[ e^{\lambda t}|p-q\xi-\zeta(r-q\eta)|<\delta, e^{(1-\lambda)t}|r-q\eta|<\delta,e^{-t}|q| <\delta. \]
Since $q\neq 0$, we have
\begin{align*}
 & \max \left\{ |q|^{\lambda}|q \xi-p-\zeta(q \eta-r)|, |q|^{1-\lambda}|q \eta-r| \right\} \\
 =~& \max \left\{ |e^{-t}q|^{\lambda}e^{\lambda t}|q \xi-p-\zeta(q \eta-r)|, |e^{-t}q|^{1-\lambda}e^{(1-\lambda)t}|q \eta-r| \right\} \\
 <~& \delta \max \left\{ \delta^{\lambda},\delta^{1-\lambda} \right\} = \max \left\{ \delta^{1+\lambda},\delta^{2-\lambda} \right\} \leq\varepsilon,
\end{align*}
contradicting \eqref{AGKe3.3}.

Conversely, suppose that there exists $0<\delta\leq 1$ so that \eqref{AGKe3.4} is true.  For a given $(p,r,q)\in \tilde{\mathcal{O}}^3_{\mathbb{K}}$, let $t_0\geq 0$ satisfy $\delta = 2e^{-t_0}|q|$.  Then, from \eqref{AGKe3.4}, we have
\begin{align*}
 & \max \left\{ |q|^{\lambda}|q \xi-p-\zeta(q \eta-r)|, |q|^{1-\lambda}|q \eta-r| \right\} \\
 \geq~& \left( \dfrac{\delta}{2} \right)^{\lambda} \max \left\{ e^{\lambda t_0}|q \xi-p-\zeta(q \eta-r)|, e^{(1-\lambda)t_0}|q \eta-r| \right\} \\
 \geq~& \delta \left( \dfrac{\delta}{2} \right)^{\lambda},
\end{align*}
so that \eqref{AGKe3.3} is true for $\varepsilon= \dfrac{\delta^{1+\lambda}}{2^\lambda}$.  This completes the proof.
\end{proof}

\begin{remark}
 The case $\lambda<\frac{1}{2}$ is readily obtained through conjugation by the permutation matrix $P$ defined in \eqref{perm}, with $u_{\mathbf{z}},\mathbf{v}$ replaced by $\tilde{u}_{\mathbf{z}},\tilde{\mathbf{v}}$.
 \end{remark}

Let $B=\mathrm{cl}(B((\xi_{B},\eta_{B},\zeta_{B}),\rho(B))\subset\mathbb{C}^3$.  By Corollary \ref{AGKl3.3}, the collection of vectors
\[ \mathscr{W}(B,\mathbf{v}) := \begin{matrix} \{ \mathbf{w}=(a,b,c)\in\mathcal{O}_{\mathbb{K}}^3:(a,b)\neq(0,0), \langle\mathbf{v},{\mathbf{w}}\rangle={0}, \\ |a|<\sqrt[4]{4d}|q|^\lambda,|b+\zeta_{B}a|<\sqrt[4]{4d}|q|^{1-\lambda}+\sqrt{\rho(B)} \} \end{matrix} \]
is nonempty for a given $\mathbf{v}=(p,r,q)\in \tilde{\mathcal{O}}^3_{\mathbb{K}}$, where $\langle\cdot,\cdot\rangle$ is defined in \eqref{hermprod}.  Denote the shortest distal vector of $\mathscr{W}(B,\mathbf{v})$ (that is, the least-norm vector among those closest to the boundary of the bounding region of $\mathscr{W}(B,\mathbf{v})$) by 
\begin{equation}\label{w}
\mathbf{w}(B,\mathbf{v}):=(a(B,\mathbf{v}),b(B,\mathbf{v}),c(B,\mathbf{v}))
\end{equation}
such that
\begin{equation}\label{AGKe3.5}
 \max\{ |a(B,\mathbf{v})|,|b(B,\mathbf{v})+\zeta_{B}a(B,\mathbf{v})| \}=\min\left\{ \max_{(a,b,c)\in\mathscr{W}(B,\mathbf{v})}\{ |a|,|b+\zeta_{B}a| \} \right\}.
\end{equation}

\begin{lemma}\label{AGKl3.4}
Define the height of $\mathbf{v}$ relative to $B$ by
\begin{equation}\label{AGK_height} H_B(\mathbf{v}):=|q|\max\{ |a(B,\mathbf{v})|,|b(B,\mathbf{v})+\zeta_{B}a(B,\mathbf{v})| \}. \end{equation} Then
\begin{equation}\label{AGKe3.6}
 |q|\leq H_B(\mathbf{v}) \leq \sqrt[4]{4d}|q|^{\max\{1+\lambda,2-\lambda\}}
\end{equation}
for any $\mathbf{v}=(p,r,q)\in \tilde{\mathcal{O}}^3_{\mathbb{K}}$ and any closed ball $B\subset\mathbb{C}^3$.
\end{lemma}
The proof is analogous to \cite[Lemma 2.4]{AGK} with Corollary \ref{AGKl3.3} in place of Lemma 3.3, and is thus omitted.

For the rest of the section, we will assume that $\lambda\geq\frac{1}{2}$.  The case $\lambda<\frac{1}{2}$ follows on appropriately substituting $\lambda$ by $1-\lambda$.

Now, let $B_0=\mathrm{cl}(B((\xi_0,\eta_0,\zeta_0),\rho_0))\subset\mathbb{C}^3$ be such that $\rho_0\leq 1$.  Set
\begin{equation}\label{AGKe3.7}
 \kappa:=1+\max\{ |\xi_0|,|\eta_0|,|\zeta_0| \}.
\end{equation}
With $0<\beta<1$, choose positive numbers $R$ and $\varepsilon$ satisfying
\begin{equation}\label{AGKe3.8}
 R\geq\max\left\{ 10^8d\kappa^6,\dfrac{4}{\beta} \right\}
\end{equation}
and
\begin{equation}\label{AGKe3.9}
 \varepsilon\leq\dfrac{\rho_0}{200d\kappa^3R^{10}},
\end{equation}
where $d$ arises from the imaginary quadratic number field $\mathbb{K}=\mathbb{Q}[\sqrt{-d}]$.  We will use the parameters $R,\varepsilon$ and the height function $H$ to define the families $\{\mathscr{B}_n\}$ of closed-ball subsets of $B_0$ and appropriate sets of integral vectors $\mathscr{V}_B$ as below:
\begin{equation}\label{closedballfamily}
 \mathscr{B}_0=\{B_0\}, \quad \mathscr{B}_n:= \left\{ B\subset B_0:\dfrac{\beta\rho_0}{R^n}<\rho(B)\leq\dfrac{\rho_0}{R^n} \right\};
\end{equation}
setting
\[ H_n=\dfrac{\varepsilon\kappa}{\rho_0}R^n\quad(n\geq 0), \]
define
\begin{equation}\label{V_B}
 \mathscr{V}_B:=\{ \mathbf{v}\in \tilde{\mathcal{O}}^3_{\mathbb{K}}:3{d}H_n\leq H_B(\mathbf{v})\leq 2H_{n+1} \}\quad\text{if~there~is}~n\in\mathbb{N}~\text{such~that}~B\in\mathscr{B}_n.
\end{equation}
The sets $\mathscr{V}_B$ are well-defined: in view of \eqref{AGKe3.8}, we have $\mathscr{B}_m\bigcap\mathscr{B}_n = \emptyset$ for $m\neq n$, so that any $B\subset B_0$ belongs to at most one $\mathscr{B}_n$.  From Lemma \ref{AGKl3.4}, it follows that if $\mathbf{v}=(p,r,q)\in \mathscr{V}_B$, then
\begin{equation}\label{AGKe3.11}
 \dfrac{\sqrt{3d}}{\sqrt[4]{2}}H_n^{\frac{1}{1+\lambda}}\leq |q| \leq 2H_{n+1}.
\end{equation}
We define
\begin{equation}\label{V_B,k}
 \left. \begin{matrix}
 \mathscr{V}_{B,1}:=\{ (p,r,q)\in\mathscr{V}_B:|q|\in \frac{\sqrt{3d}}{\sqrt[4]{2}}H_n^{\frac{1}{1+\lambda}}[1,R^8] \}, \quad\text{and}\\
\mathscr{V}_{B,k}:=\{ (p,r,q)\in\mathscr{V}_B: |q|\in \frac{\sqrt{3d}}{\sqrt[4]{2}}H_n^{\frac{1}{1+\lambda}}[R^{2k+4},R^{2k+6}] \} \quad(k\geq 2).
\end{matrix} \right\}
\end{equation}

Exactly analogous to \cite[Lemma 3.5]{AGK}, the following partitioning holds:

\begin{lemma}\label{AGKl3.5}
 If $B\in\mathscr{B}_n$, then $\mathscr{V}_B=\bigcup_{k=1}^n \mathscr{V}_{B,k}$.
\end{lemma}

Now we define subfamilies $\mathscr{B}'_n\subset\mathscr{B}_n$.  Set $\mathscr{B}'_0=\mathscr{B}_0$.  Whenever $\mathscr{B}'_{n-1}$ is defined, set
\begin{equation}\label{AGKe3.12}
 \mathscr{B}'_n := \left\{ B\in\mathscr{B}_n:B\subset B'~\text{for~some}~B'\in \mathscr{B}'_{n-1}~\text{and}~B\cap\bigcup_{\mathbf{v}\in\mathscr{V}_B} \Delta_{\varepsilon,\mathbb{K},\lambda}(\mathbf{v})=\emptyset \right\}.
\end{equation}

Propositions \ref{AGKl3.6} and \ref{AGKl3.7} describe the necessary properties of the modified subfamilies $\mathscr{B}'_n$:

\begin{proposition}\label{AGKl3.6}
 Let $n\geq 0$, $B\in\mathscr{B}'_n$, $\mathbf{v}=(p,r,q)\in \tilde{\mathcal{O}}^3_{\mathbb{K}}$.  If $|q|^{1+\lambda}\leq 2H_{n+1}$, then $\Delta_{\varepsilon,\mathbb{K},\lambda}(\mathbf{v})\,\cap\, B = \emptyset$.
\end{proposition}
\begin{proof}
 Using \eqref{AGK_height}, \eqref{AGKe3.8}, \eqref{AGKe3.9} and \eqref{AGKe3.12}, and with the replacements
 \begin{equation}\label{replacements}
  \left. \begin{matrix}
          (x,y,z)\in\mathbb{R}^3 \to (\xi,\eta,\zeta)\in\mathbb{C}^3,\quad (\lambda,\mu)\to(\lambda,1-\lambda), \\
          (p,r,q)\in\mathbb{Z}^2\times\mathbb{N} \to (p,r,q)\in\tilde{\mathcal{O}}_{\mathbb{K}}^3, \\
        \text{absolute~norm}~|\cdot|~\text{on}~\mathbb{R}^3\to~\text{absolute~norm}~|\cdot|~\text{on}~\mathbb{C}^3 \\
          \text{(in~particular,}~q\to|q|\text{)},
         \end{matrix} \right\}
 \end{equation}
the proof follows analogously to that of \cite[Lemma 3.6]{AGK}.
\end{proof}

The following estimates hold analogous to \cite[Lemma 5.1]{AGK}:

\begin{lemma}\label{AGKl5.1}
 Let $n\geq 0$, $B\in\mathscr{B}'_n$, $k\geq 1$, and $B_j\in\mathscr{B}_{n+k}$, $j=1,2$ are such that $B_j\subset B$.  Let $\mathbf{v}_j=(p_j,r_j,q_j)\in\mathscr{V}_{B,k}$ be such that $B\cap \Delta_{\varepsilon, \lambda}(\mathbf{v}_j)\neq\emptyset$, and $\mathbf{w}_j=\mathbf{w}(B_j,\mathbf{v}_j)$.  Then
 \begin{equation}\label{AGKe5.1-5.2}
  |\langle\mathbf{v}_j,\mathbf{w}_{j'}\rangle|\leq 18\varepsilon\kappa\sqrt{d} R^{d_k}+64\varepsilon\kappa^3R^{k+1}\dfrac{|q_j|}{|q_{j'}|} \quad (j,j'\in\{1,2\}, j\neq j'),
 \end{equation}
 where \[ d_k=\begin{cases}
               8 & \text{if}~k=1,\\
               2 & \text{if}~k>1.
              \end{cases}
 \]
\end{lemma}
For a closed ball $B\subset\mathbb{C}^3$ and $\mathbf{v}=(p,r,q)\in \tilde{\mathcal{O}}^3_{\mathbb{K}}$, we also consider the complex hyperplane
\[ \mathscr{H}(B,\mathbf{v})=\{(\xi,\eta,\zeta)\in\mathbb{C}^3:\langle (\xi,\eta,1),{\mathbf{w}}\rangle=0\}, \]
where $\mathbf{w}=\mathbf{w}(B,\mathbf{v})$ is given by \eqref{w}.  Then, analogous to \cite[Lemma 5.2]{AGK}, Lemma \ref{AGKl5.1} implies the following:
\begin{lemma}\label{AGKl5.2}
 Let $n\geq 0$, $B\in\mathscr{B}'_n$, $k\geq 1$.  Then one of the following statements is true:
 \begin{enumerate}
  \item There exists a complex hyperplane $\mathscr{H}_k(B)$ such that for any $B'\in\mathscr{B}_{n+k}$ with $B'\subset B$, \[ \mathbf{v}\in\mathscr{V}_{B',k},\, B\cap\Delta_{\varepsilon,\mathbb{K},\lambda}(\mathbf{v})\neq\emptyset\,\implies\,\mathscr{H}(B',\mathbf{v})=\mathscr{H}_k(B), \]
  \item $k=1$, and there exists $\mathbf{v}_0=(p_0,r_0,q_0)\in \tilde{\mathcal{O}}^3_{\mathbb{K}}$ with $\frac{\sqrt{3d}}{\sqrt[4]{2}}H_{n+1}^{\frac{1}{1+\lambda}}\leq|q_0|\leq 2H_{n+2}$ such that for any $B'\in\mathscr{B}_{n+1}$ with $B'\subset B$,
  \[ \mathbf{v}\in\mathscr{V}_{B',1},\,B\cap\Delta_{\varepsilon,\mathbb{K},\lambda}(\mathbf{v})\neq\emptyset\,\implies\mathbf{v}=t\mathbf{v}_0~\text{for some}~t\in\mathbb{C}~\text{with}~|t|\geq1. \]
 \end{enumerate}
\end{lemma}
\begin{remark}
The key point in proving Lemma \ref{AGKl5.2} is to show that $\mathbf{w}_1,\mathbf{w}_2$ are linearly dependent, where 
\[ B_j\in\mathscr{B}_{n+k}, B_j\subset B, \mathbf{v}_j\in\mathscr{V}_{B_j,k},B\cap\Delta_{\varepsilon,\mathbb{K},\lambda}(\mathbf{v_j})\neq\emptyset, \mathbf{w}_j=\mathbf{w}(B_j,\mathbf{v}_j) \quad(j=1,2). \]
This is achieved by using the complex vector triple product identity \eqref{vec3prod} and the fact that $\mathbf{v}_1\cdot\mathbf{w}_1=0$ (from the definition of $\mathscr{W}(B_1,\mathbf{v}_1$)) to get
\[ \mathbf{v}_1\times(\mathbf{w}_1\times\mathbf{w}_2) = (\mathbf{v}_1\cdot\mathbf{w}_2)\mathbf{w}_1. \]
The estimates from Lemma \ref{AGKl5.1} show that $\mathbf{w}_1\times\mathbf{w}_2=\mathbf{0}$, yielding the linear dependence of $\mathbf{w}_1$ and $\mathbf{w}_2$.
\end{remark}

For a complex hyperplane $\mathscr{H}\in\mathbb{C}^n$ and $\delta>0$, denote the $\delta$-neighborhood of $\mathscr{H}$ by
\begin{equation}\label{hypernbd}
\mathscr{H}^{(\delta)}:=\{\mathbf{z}\in\mathbb{C}^n:\inf_{\mathbf{w}\in \mathscr{H}_j}\|\mathbf{z}-\mathbf{w}\|<\delta_j\}.
\end{equation}
At this stage, we will prove that, given a pair $(B,\mathbf{v})$, there exist complex hyperplanes $\{\mathscr{H}_k(B)\}_{k\in\mathbb{N}}$ whose appropriate neighborhoods fully contain the intersections of $\Delta_{\varepsilon,\mathbb{K},\lambda}(\mathbf{v})$ with the admissible balls contained in $B$: 

\begin{proposition}\label{AGKl3.7}
 Let $n\geq 0$, $B\in\mathscr{B}'_n$, and $k\geq 1$.  Then there exists a complex hyperplane $\mathscr{H}_k(B)\subset\mathbb{C}^3$ such that
 \[ \Delta_{\varepsilon,\mathbb{K},\lambda}(\mathbf{v})\cap B'\subset\mathscr{H}_k(B)^{(\rho_0R^{-n-k}d^{-1/2})} \]
 for any $B'\in\mathscr{B}_{n+k}$ with $B'\subset B$ and any $\mathbf{v}\in\mathscr{V}_{B',k}$.
\end{proposition}

\begin{proof}
 The case of $B$ avoiding $\Delta_{\varepsilon,\mathbb{K},\lambda}(\mathbf{v})$ is trivial.  Lemma \ref{AGKl5.2} is applicable otherwise.
 
 Let $B'\in\mathscr{B}_{n+k}$ and $\mathbf{v}=(p,r,q)\in\mathscr{V}_{B'}$.  In Case (1) of Lemma \ref{AGKl5.2}, it follows that
 \[ \rho(B')|q|\leq\dfrac{2H_{n+k+1}\rho_0}{R^{n+k}}=2\varepsilon\kappa R\leq\dfrac{1}{4}. \]
 Denote $\bar{\mathbf{w}}(B',\mathbf{v})=(a,b,c)$.  Then for $\mathbf{z}=(\xi,\eta,\zeta)\in \Delta_{\varepsilon,\mathbb{K},\lambda}(\mathbf{v})\cap B'$, we have
 \begin{align*}
  |\langle(\xi,\eta,1),{\mathbf{w}}(B',\mathbf{v})\rangle| &= |a\xi+b\eta+c|\\
  &= \left| a\left( \xi-\dfrac{p}{q}-\zeta\left( \eta-\dfrac{r}{q} \right) \right) + (b+\zeta a)\left( \eta-\dfrac{r}{q} \right) \right| \\
  &\leq |a|\dfrac{\varepsilon\sqrt[4]{d} }{|q|^{1+\lambda}} + |b+\zeta a|\dfrac{\varepsilon\sqrt[4]{d} }{|q|^{2-\lambda}}\\
  &\leq \dfrac{\varepsilon }{|q|}\left( \dfrac{|a|}{|q|^\lambda} + \dfrac{|b+\zeta_{B'}a|}{|q|^{1-\lambda}} + \dfrac{|a||\zeta-\zeta_{B'}|}{|q|^{1-\lambda}} \right)\\
  &\leq \dfrac{\varepsilon\sqrt[4]{d} }{|q|}\left( 2\sqrt[4]{4d} + \dfrac{\sqrt{\rho(B')}}{|q|^{1-\lambda}} + \rho(B')|q|^{2\lambda-1} \right)\\
  &\leq \dfrac{\varepsilon\sqrt[4]{d} }{|q|}\left( 2\sqrt[4]{4d} + \sqrt{\dfrac{\rho(B')}{|q|}} + \sqrt[4]{4d}\rho(B')|q| \right)\\
  &\leq \dfrac{3\varepsilon\sqrt{2d}}{|q|}.
 \end{align*}
 Since $\displaystyle H_{B'}(\mathbf{v}) = |q|\max\{|a|,|b+\zeta_{B'}a|\}\leq\kappa|q|\max\{|a|,|b|\}$, the Euclidean distance of the point $\mathbf{z}$ from the plane $\mathscr{H}(B',\mathbf{v})$ is
 \[
  \dfrac{|a\xi+b\eta+c|}{\sqrt{|a|^2+|b|^2}} < \dfrac{3\varepsilon\sqrt{2d}}{|q|\max\{|a|,|b|\}} \leq \dfrac{3\varepsilon\kappa\sqrt{2d}}{H_{B'}(\mathbf{v})} \leq \dfrac{\varepsilon\kappa}{\sqrt{d}H_{n+k}} = \dfrac{\rho_0}{\sqrt{d}R^{n+k}}.
 \]
 Therefore, in Case (1), $\mathbf{z}\in\mathscr{H}(B',\mathbf{v})^{(\rho_0R^{-n-k}d^{-1/2})}$; that is, \[ \Delta_{\varepsilon,\mathbb{K},\lambda}(\mathbf{v})\cap B'\subset\mathscr{H}_k(B)^{(\rho_oR^{-n-k}d^{-1/2})},\] where $\mathscr{H}_k(B)=\mathscr{H}(B',\mathbf{v})$.

 In Case (2), we show that $\Delta_{\varepsilon,\mathbb{K},\lambda}(\mathbf{v})\cap B\subset\mathscr{H}_1(B)^{(\rho_0R^{-n-1}d^{-1/2})}$ for $\mathbf{v}=t\mathbf{v}_0$ with $|t|\geq 1$, where
 \[
 \mathscr{H}_1(B)=\left\{ (\xi,\eta,\zeta)\in\mathbb{C}^3:\xi-\dfrac{p_0}{q_0}-\zeta_B\left(\eta-\dfrac{r_0}{q_0} \right)=0 \right\}
 \]
 for $\mathbf{v}_0=(p_0,r_0,q_0)$.  If $\mathbf{v}=(p,r,q)$, then $|q|\geq |q_0|$, and $\dfrac{p}{q}=\dfrac{p_0}{q_0}$, $\dfrac{r}{q}=\dfrac{r_0}{q_0}$.  Then any $\mathbf{z}=(\xi,\eta,\zeta)\in \Delta_{\varepsilon,\mathbb{K},\lambda}(\mathbf{v})\cap B$ satisfies
 \begin{align*}
  \dfrac{1}{\sqrt{1+\zeta_B^2}}\left| \xi-\dfrac{p_0}{q_0}-\zeta_B\left(\eta-\dfrac{r_0}{q_0} \right) \right| &\leq \left| \xi-\dfrac{p}{q}-\zeta\left(\eta-\dfrac{r}{q} \right) \right| + |\zeta-\zeta_B|\left| \eta-\dfrac{r}{q} \right| \\
  &\leq \dfrac{\varepsilon\sqrt[4]{d} }{|q|^{1+\lambda}} + \rho(B)\dfrac{\varepsilon\sqrt[4]{d} }{|q|^{2-\lambda}}\\
  &\leq \dfrac{\varepsilon\sqrt[4]{d} }{|q_0|^{1+\lambda}} + \rho(B)\dfrac{\varepsilon\sqrt[4]{d} }{|q_0|^{2-\lambda}}\\
  &\leq \dfrac{\varepsilon\sqrt[4]{d} }{|q_0|^{1+\lambda}}(1+\rho(B)|q_0|)\\
  &\leq \dfrac{\varepsilon\sqrt[4]{2}}{\sqrt{3d}H_{n+1}}\left( 1+\dfrac{2H_{n+2}\rho_0}{R^n} \right)\\
  &= \dfrac{\rho_0\sqrt[4]{2}(1+2\varepsilon\kappa R^2)}{\sqrt{3d}\kappa R^{n+1}} <\dfrac{\rho_0}{\sqrt{d}R^{n+1}};
 \end{align*}
 so that $\mathbf{z}\in\mathscr{H}_1(B)^{(\rho_0R^{-n-1}d^{-1/2})}$.  Thus $\Delta_{\varepsilon,\mathbb{K},\lambda}(\mathbf{v})\cap B\subset\mathscr{H}_1(B)^{(\rho_0R^{-n-1}d^{-1/2})}$ in Case (2) as well.
\end{proof}

\section{Schmidt's Absolute Games in Complex Euclidean Spaces}
We refer to \cite{Sch22,Sch23} and \cite{BFKRW,FSU,KW17} for the basic notions of Schmidt's Game and some of its variants.  The \emph{hyperplane-absolute} (HA) game in $\mathbb{R}^n$ has been studied in depth in \cite{BFKRW,FSU,KW17}.  The first instance of such a game in $\mathbb{C}^n$ appeared in \cite{S}, which we recall here for the sake of completeness.

In what follows, we consider the complex Euclidean space $\mathbb{C}^n$ and set up an HA game as follows:
\begin{itemize}
 \item Two players, Player A and Player B, begin the game with a parameter $0<\beta<\frac{1}{3}$.
 \item Player B initiates by choosing a closed ball $B_0\subset\mathbb{C}^n$ of radius $\rho=\rho_0$.
 \item With Player B having chosen the closed ball $B_j$ with radious $\rho_j$ as their $j^{\text{th}}$ move, Player A chooses a complex hyperplane $\mathscr{H}_j$ and marks a hyperplane-neighborhood
 $\mathscr{H}_j^{(\delta_j)}$ defined in \eqref{hypernbd}
 with $\delta_j\leq\beta\rho_j$ as their $j^{\text{th}}$ move.
 \item The $(j+1)^{\text{th}}$ move by Player B is a closed ball $B_{j+1}\subset B_j\setminus \mathscr{H}_j^{(\delta_j)}$ of radius $\rho_{j+1}\geq\beta\rho_j$.
\end{itemize}
A sequence of nested closed balls $\{B_j\}_{j\in\mathbb{N}}$ results from the above procedure.

A $\beta$-hyperplane-absolute-winning (i.e. $\beta$-HAW) subset $E\subset\mathbb{C}^n$ is one where for every possible sequence of moves $\{B_j\}_{j\in\mathbb{N}}$ by Player B, Player A can force $E\bigcap\left(\cap_{j\in\mathbb{N}}B_j\right)\neq\emptyset$.  Such a set $E$ is called hyperplane-absolute-winning (i.e. HAW) if it is $\beta$-HAW for any $0<\beta<\frac{1}{3}$.  Some properties of $\beta$-HAW and HAW sets in $\mathbb{C}^n$ are listed in \cite[Lemma 2.1]{S}.

On the other hand, the \emph{hyperplane-potential} (HP) game on $\mathbb{C}^n$ uses two fixed parameters $0<\beta<1$ and $\gamma>0$.  In this game, the $j^{\text{th}}$ move by Player A consists of a countable collection of hyperplane neighborhoods $\{\mathscr{H}^{(\delta_{j,k})}_{j,k}:k\in\mathbb{N}\}$ satisfying 
 \[ 
 \sum_{k\in\mathbb{N}} \delta_{j,k}^\gamma\leq(\beta\rho_j)^\gamma; 
 \] 
 while the rest of the game follows exactly like the HA game described above.  The notions of $(\beta,\gamma)$-hyperplane-potential-winning (i.e., $(\beta,\gamma)$-HPW) and HPW sets are analogous to their HAW counterparts: in particular, $E\subset\mathbb{C}^n$ is $(\beta,\gamma)$-HPW if Player A can always guarantee that \[ \bigcap_{j=0}^\infty B_j\cap\left( E\cup\bigcup_{j=0}^\infty \bigcup_{k\in\mathbb{N}} \mathscr{H}_{j,k}^{(\delta_{j,k})}\right) \neq\emptyset,\]
 no matter how Player B chooses their sequence of moves $\{ B_j \}_{j\in\mathbb{N}}$.

{\begin{remark}\label{rem}
As shown in \cite[Lemma 2.4]{NS}, the notions of HAW and HPW sets are equivalent.  Thus, if we can show a certain set to be HPW, we are assured that the same set is HAW as well.
\end{remark}
}
\begin{proposition}\label{bad}
 The set $\mathrm{B}_{\mathbb{K}}(\lambda)$ in \eqref{bad-lambda} is HAW.
\end{proposition}
\begin{proof}
 We set up an HP game in $\mathbb{C}^3$ with the target set being $\mathrm{B}_{\mathbb{K}}(\lambda)$ as in \eqref{bad-lambda}.  Let $0<\beta<1$, $\gamma>0$.  Let Player B make the initial move as a closed ball $B_0\subset\mathbb{C}^3$.  Assume that the moves by Player B satisfy $\rho_j=\rho(B_j)\to 0$; otherwise Player A can win by default.  With possibly arbitrary initial moves by Player A and relabeling indices, assume that $\rho_0\leq 1$.  In view of the ruleset of the HP games, assume that $R$ defined in \eqref{AGKe3.8} additionally satisfies
 \begin{equation}\label{AGKe3.13} \dfrac{1}{R^\gamma-1}\leq\left( \dfrac{\beta^2}{2} \right)^\gamma. \end{equation}
 Then $\varepsilon$, $\mathscr{B}_n$, $\mathscr{V}_B$, $\mathscr{V}_{B,k}$, and $\mathscr{B}'_n$ defined as in \eqref{AGKe3.9}, \eqref{closedballfamily}, \eqref{V_B}, \eqref{V_B,k}, and \eqref{AGKe3.12}, respectively, imply the truth of Lemma \ref{AGKl3.5} and Propositions \ref{AGKl3.6} and \ref{AGKl3.7}.  In response of Player B making the move $B_j$, Player A can utilize the strategy outlined below: 
 
 \begin{enumerate}
  \item If there is $n\geq 0$ such that $j$ is the smallest index for which $B_j\in \mathscr{B}'_n$, then choose the family of complex hyperplane neighborhoods $\{ \mathscr{H}_k(B)^{(2\rho_0R^{-n-k}d^{-1/2})}:k\in\mathbb{N} \}$ with the planes $\mathscr{H}_k(B)$ are as in Proposition \ref{AGKl3.7}.
  \item If no such $n$ exists at the $j^{\text{th}}$ move, then choose an arbitrary complex hyperplane neighborhood avoiding $B_j$.
 \end{enumerate}
Let $j_n$ be the index $j$ for which (1) above holds, and let $\mathscr{N}$ be the set of all such $n$.  This move by Player A is legal, since we have $B_{j_n}\in\mathscr{B}_n$, so that $\rho_{j_n}>\dfrac{\beta\rho_0}{R^n}$, and
\[ \sum_{k\in\mathbb{N}}\left(\dfrac{2\rho_0}{R^{n+k}} \right)^\gamma = \left(\dfrac{2\rho_0}{R^{n}} \right)^\gamma \dfrac{1}{R^\gamma-1}\leq(\beta\rho_0)^\gamma \]
then follows from \eqref{AGKe3.13}.  Then Player A has a guaranteed win with the above strategy; that is,
\[ \bigcap_{j\in\mathbb{N}} B_j \subset \mathrm{B}_{\mathbb{K}}(\lambda) \cup\bigcup_{n\in\mathscr{N}}\bigcup_{k\in\mathbb{N}} \mathscr{H}_k(B_{j_n})^{(2\rho_0R^{-n-k}d^{-1/2})}. \]
Analogous to \cite{AGK}, the case $\mathscr{N}\neq \mathbb{N}\cup\{0\}$ follows on applying Lemma \ref{AGKl3.5} to the smallest $n\notin\mathscr{N}$, and then Proposition \ref{AGKl3.7} to $B=B_{j_{n-k}}$ and $B'=B_j$, where $j$ is such that $B_j\in\mathscr{B}_n$ but $B_{j-1}\notin\mathscr{B}_j$; on the other hand, the case $\mathscr{N}= \mathbb{N}\cup\{0\}$ is a consequence of Proposition \ref{AGKl3.6}.  Hence Player A always wins with the target set being $\mathrm{B}_{\mathbb{K}}(\lambda)$, proving Proposition \ref{bad}.
\end{proof}
 
 \begin{corollary}\label{cor3.3}
Let $0\leq \lambda\leq 1$, let $F_\lambda^+$ and $U$ be as in \eqref{F-lambdaplus} and \eqref{U}, respectively.  Then the set \[ \{u\in U:u\Gamma\in E(F_\lambda^+)\} \] is HAW in $U$. 
 \end{corollary}
The proof follows from Proposition \ref{AGKl3.2} and the biholomorphic correspondence between complex manifolds proved in \cite[Proposition 5.1]{S}.

\begin{proof}[Proof of Theorem \ref{EF}]
Using \cite[Proposition 5.1]{S} in place of \cite[Lemma 2.2(iii,v)]{AGK}, \cite[Lemma 2.2(3)]{S} instead of \cite[Lemma 2.2(iv)]{AGK}, and Corollary \ref{cor3.3} for \cite[Theorem 3.1]{AGK}, the proof of Theorem \ref{EF} is exactly analogous to that of \cite[Theorem 1.2, Theorem 1.3]{AGK}.
\end{proof}

\section*{Acknowledgements}
I sincerely thank Prof. Anish Ghosh for suggesting the question, for helpful discussions and for guidance throughout the course of this work.  I extend my sincere gratitude to the anonymous referee for their insightful comments and suggestions.  I am particularly thankful for suggesting the corrected version of Corollary \ref{AGKl3.3} along with the idea of the proof which has significantly improved the result.  This work is supported by CSIR-JRF grant 09/877(0014)/2019-EMR-I.

\end{document}